\documentclass{article}

\usepackage{amsfonts}
\usepackage{xypic}
\usepackage{amssymb}

\usepackage{mathtools}

\input xy
\title{Second Countable Virtually Free Pro-$p$ Groups whose Torsion Elements have Finite Centralizer}

\author{J.W. MacQuarrie\footnote{Universidade Federal de Minas Gerais, john@mat.ufmg.br, corresponding author}, P.A. Zalesskii\footnote{Universidade de Bras\'\i lia, pz@mat.unb.br}
\thanks{Research partially supported by the Conselho Nacional de Desenvolvimento Cient\'\i fico e Tecnol\'ogico (CNPq)}}

\newif\iftest % Testversion / Finalversion
\testfalse %%%%%%%%%%%%%%%%%%%%%%%%%%%%%%%%%%%%%%%%%%%%%%%%TESTTEST
\newif\ifteston
\iftest\testontrue\else\testonfalse\fi
\newif\ifnotes % notes enabled
\notestrue
\newif\iffinal
\finalfalse
\iffinal\global\notesfalse\global\testfalse\fi  %
\newcounter{notes}
0

\newcommand{\ignoriere}[1]{}

\def\plabel#1{\label{#1}\iftest\fbox{#1 }\fi}
\def\pref#1{\ref{#1}\iftest{ \fbox{#1 }}\fi}
\def\pcite#1{\cite{#1}\iftest{\fbox{#1 }}\fi}
\newcommand{\prefeq}[1]{Eq.(\pref{e-#1})}

\newcommand{\preflemma}[1]{Lemma \pref{l-#1}}
\newcommand{\prefdef}[1]{Definition \pref{d-#1}}
\newcommand{\prefprop}[1]{Proposition \pref{p-#1}}

\newtheorem{theorem}{Theorem}[section]
\newtheorem{lemma}[theorem]{Lemma}
\newtheorem{definition}[theorem]{Definition}

\newtheorem{proposition}[theorem]{Proposition}

\newtheorem{corollary}[theorem]{Corollary}
\newtheorem{remark}[theorem]{Remark}

\newcommand{\res}[1]{\hspace{-0.6mm}\downarrow_{\hspace{-0.25mm}{#1}}}
\newcounter{claims}

\newcommand{\bd}[1]{\begin{definition}\plabel{d-#1}\rm}
\newcommand{\ed}{\end{definition}}
\newcommand{\bt}[1]{\begin{theorem}\plabel{t-#1}\setcounter{claims}0}
\newcommand{\et}{\end{theorem}}
\newcommand{\bl}[1]{\begin{lemma}\plabel{l-#1}\setcounter{claims}0}
\newcommand{\el}{\end{lemma}}
\newcommand{\bc}[1]{\begin{corollary}\plabel{c-#1}}
\newcommand{\ec}{\end{corollary}}
\newtheorem{notation}{Notation}
\newcommand{\bn}[1]{\begin{notation}\plabel{n-#1}\rm}
\newcommand{\en}{\end{notation}}
\newcommand{\brem}[1]{\begin{remark}\plabel{rem-#1}\rm}
\newcommand{\erem}{\end{remark}}
\newcommand{\bp}[1]{\begin{proposition}\plabel{p-#1}\setcounter{claims}0}
\newcommand{\ep}[1]{\end{proposition}}
\newcommand{\be}[1]{\begin{equation}\plabel{e-#1}}
\newcommand{\ee}{\end{equation}}
\newlength{\help}
\newcounter{subclaims}
\newenvironment{env}{}{}

\newcommand{\bcl}{\medskip\addtocounter{claims}1\setcounter{subclaims}0{
\noindent{\it Claim \arabic{claims}: } }
                  \begin{env}\it}

\newcommand{\ecl}{\end{env}\rm\medskip}
\newcommand{\bscl}
  {\addtocounter{subclaims}1{Subclaim \arabic{subclaims}}\begin{env}\it}
\newcommand{\escl}{\end{env}\rm\smallskip}

\newcommand{\claimproof}{{}}

\newcounter{inrmlist}
\newcounter{inalphlist}
\newenvironment{rmenumerate}{
                       \setcounter{inrmlist}{0}
                       \begin{list}
                         {\makebox[2em]{(\roman{inrmlist})}}{
                         \addtolength{\leftmargin}{1em}
                         \addtolength{\itemsep}{1mm}
                         \addtolength{\itemindent}{-2em}\usecounter{inrmlist}}
                       }%
                       {\end{list}}

                       {\end{list}}

\newcommand{\N}{{\bf N}}
\newcommand{\tF}{\tilde F}
\newcommand{\tG}{\tilde G}

\newcommand{\tv}{\tilde v}
\newcommand{\tw}{\tilde w}

\newcommand{\Z}{\mathbb{Z}}
\newcommand{\F}{\mathbb{F}}

\newcommand{\cex}{counter-example}
\newcommand{\epi}{epimorphism}

\newcommand{\fpp}{free \pp}

\newcommand{\fppgrp}{free \ppgrp}

\newcommand{\pgraph}{profinite graph}

\newcommand{\pp}{pro-$p$}

\newcommand{\ppgrp}{\pp\ group}

\newcommand{\mono}{monomorphism}
\newcommand{\vfpp}{virtually \fpp}
\newcommand{\vfppgrp}{\vfpp\ group}

 \newcommand{\pCgrp}{pro-$p$ group}

\newcommand{\cC}{{\cal C}}

\newcommand{\ind}[1]{\hspace{-0.6mm}\uparrow^{\hspace{-0.25mm}{#1}}}
\newcommand{\ctens}{\widehat{\otimes}}

\newcommand{\onto}{\twoheadrightarrow}

\newcommand{\Hom}{\textnormal{Hom}}

\newcommand{\tn}[1]{\textnormal{#1}}

\newcommand{\path}[2]{[#1,#2]}                %  path beginning at 1 ending at 2

\newcommand{\inv}{^{-1}}                        % inverse of an element
\newcommand{\limproj}[1]{{\lim\limits_{\displaystyle\longleftarrow}}_{#1}}
\newcommand{\tor}[1]{\text{Tor}(#1)}             % Torsion
\newcommand{\gp}[1]{\langle #1 \rangle}             % Group generated by #1
\newcommand{\torgp}[1]{\gp{\tor{#1}}}              % Group generated by torsion
\newcommand{\torfactor}[1]{#1/\torgp{#1}}           % factorgp G/ tor G

\newcommand{\ugp}{\{1\}}                          % Trivial group
\newcommand{\cG}{{\cal G}}

\newcommand{\cK}{{\cal K}}

\newcommand{\C}{{\cal C}}

\newcommand{\multip}{g_{\tw}\inv}

\newcount\names
\def\aname#1#2{\ifodd\names{#1.\,#2}\else{#2 #1.}\fi}

\newcommand{\segal}{\aname{D}{Segal}}

\newcommand{\shalev}{\aname{A}{Shalev}}

\def\cK{{\cal K}}

\def\Fin#1{{\it Fin(#1)}}
\def\HNN{{\it HNN}}
\def\HNNgrp{HNN-group}
\def\HNNext{HNN-extension}

\def\KST{Kurosh subgroup theorem}

\newcommand{\homo}{homomorphism}

\def\KST{Kurosh subgroup theorem}

\begin{document}
\maketitle



\begin{abstract}
A second countable \vfppgrp\ all of whose torsion elements have
finite centralizer is the \fpp\ product of finite $p$-groups and a
\fpp\ factor.  The proof explores a connection between $p$-adic
representations of finite $p$-groups and virtually free pro-$p$
groups.  In order to utilize this connection, we first prove a
version of a remarkable theorem of A. Weiss for infinitely
generated profinite modules that allows us to detect freeness of
profinite  modules.  The proof now proceeds using techniques
developed in the combinatorial theory of profinite groups. Using
an HNN-extension we embed our group into a semidirect product
$F\rtimes K$ of a free pro-$p$ group $F$ and a finite $p$-group
$K$ that preserves the conditions on centralizers and such that
every torsion element is conjugate to an element of $K$. We then
show that the $\Z_pK$-module $F/[F,F]$ is free using the detection
theorem mentioned above. This allows us to deduce the result for
$F\rtimes K$, and hence for our original group, using the pro-$p$
version of the Kurosh subgroup theorem.
\end{abstract}

\section{Introduction}

The objective of this paper is to give a complete description of a
second countable virtually free pro-$p$ group whose torsion
elements have finite centralizer. The description is a
generalization of the main theorem of \cite{bulletin}, where the
result was obtained in the finitely generated case. Note however
that finite generation is a rather restrictive condition in Galois
theory and that the finite centralizer condition for torsion
elements arises naturally in the study of maximal pro-$p$ Galois
groups. In particular, D.\,Haran \cite{H93} (see also I.\,Efrat in
\cite{E96} for a different proof, or \cite[Proposition
19.4.3]{efrat}) proved Theorem \ref{main} in the case where $G$ is
an extension of a free pro-$2$ group with a group of order $2$.

Our main result is the following:

 \begin{theorem}\label{main}
Let $G$ be a second countable \vfppgrp\ such that the centralizer of
every torsion element in $G$ is finite. Then $G$ is a \fpp\
product of subgroups that are finite or \fpp.
\end{theorem}

The proof of Theorem \ref{main} explores a connection between
$p$-adic representations of finite $p$-groups and virtually free
pro-$p$ groups.  One of the main ingredients of the proof is the
following result, which can be considered to be a first step
towards a generalization of a remarkable theorem of A. Weiss
\cite[Theorem 2]{weiss} to infinitely generated pro-$p$ modules.

\begin{theorem}\label{infgenfree}
Let $G$ be a finite $p$-group and let $U$ be a profinite $\Z_p
G$-lattice.  Suppose that there is a normal subgroup $N$ of $G$
such that
 \begin{itemize}
  \item The restriction $U\res{N}$ is a free $\Z_p N$-module,
  \item The module $U^N$ of $N$-fixed points is a free $\Z_p[G/N]$-module.
 \end{itemize}
Then $U$ itself is a free $\Z_pG$-lattice.
\end{theorem}

\medskip

The connection to  representation theory cannot be used in a
straightforward way, however.  Indeed, if one factors out the
commutator subgroup of a free open normal subgroup $F$ then the
$G/F$-module one obtains is, in general, not a free module.  In
order to apply Theorem \ref{infgenfree}, we first use
 \pp\ \HNNext s to embed $G$ into a rather special \vfppgrp\ $\tilde G$,
in which, after factoring out the commutator of a free open normal
subgroup $\tilde F$, the obtained $\tilde G/\tilde F$-module is
free. With the aid of this module we prove Theorem \ref{main} for $\tilde G$ and
apply the Kurosh subgroup theorem to deduce the result for $G$.

\medskip
We note also that the second countability condition is essential:
for bigger cardinality the result is not true, a counter example
is available in \cite[Section 4]{HZ}.  We use notation for
profinite and pro-$p$ groups from \cite{RZ12010}.

\section{Modules}

We prove in this section a detection theorem for infinitely generated profinite permutation modules over finite $p$-groups.  Theorem \ref{infgen} is
inspired by Theorem 2 of the article \cite{weiss} by A. Weiss and our proof follows his closely.  The theorems in this section are of independent interest.

\medskip

Let $G$ be a finite group and $R$ a profinite commutative ring.  A profinite permutation $RG$-lattice $U$ is a profinite $RG$-lattice
(that is, a profinite $RG$-module that is free as a profinite $R$-module) with a profinite $R$-basis left invariant under the action of $G$ on $U$.

A profinite permutation  module is an inverse limit of permutation
modules of finite rank (this can be seen by applying \cite[Lemma
5.6.4]{RZ12010} to the profinite $R$-basis of $U$), but this does
not immediately tell us much about the structure of $U$ from a
practical point of view. For instance, there exist profinite
modules for finite rings that do not possess indecomposable
summands.  So before we begin our investigation of permutational
$\Z_pG$-lattices, we provide the general Theorem \ref{limits are
products}, which tells us in particular that a profinite permutation
lattice is simply a product of indecomposable finitely generated
permutation lattices.

We fix some notation that will remain in force for Proposition \ref{AddM equivalent to ProjE} and Theorem \ref{limits are products}.  Let $R$ be a
commutative local noetherian profinite ring, let $G$ be a finite group and let $\{M_1,\dots,M_s\}$ be a finite set of non-isomorphic finitely generated
indecomposable left $RG$-modules.  Set $M=\bigoplus_{i\in\{1,\dots, s\}}M_i$ and let $E = \tn{End}_{RG}(M)$ be the endomorphism ring of $M$.  The category $\tn{add}(M)$ has objects all finite direct sums of direct summands of $M$ and morphisms as in the full module category.  The category $\limproj{} \tn{add}(M)$ has objects all inverse limits of inverse systems in $\tn{add}(M)$.  The full subcategory $\tn{Add}(M)$ of $\limproj{} \tn{add}(M)$ consists of all modules of the form
$$\bigoplus_{i\in \{1,\dots,s\}}\prod_{\kappa_i}M_i$$
where $\kappa_i$ is a cardinal number.  Finally, denote by $\tn{Proj}(E)$ the category of profinite projective \textit{right} $E$-modules.

The following proposition is dual to \cite[Proposition 2.1]{symondsstructure}, but since duality over the given coefficient rings can be rather complicated, we choose to present a direct proof.

\bp{}\label{AddM equivalent to ProjE} The categories $\tn{Add}(M)$ and $\tn{Proj}(E)$ are equivalent. \ep

\begin{proof}
We make frequent use of how \tn{Hom} commutes with limits, products and sums, see \cite[Lemma 5.1.4]{RZ12010} (or more explicitly, \cite[\S 2.3]{symondspermcom1}).  Throughout, the notation $(X,Y)$ will be used as short-hand for $\tn{Hom}_{RG}(X,Y)$.  Define a functor $\Gamma:\tn{Add}(M)\to \tn{Proj}(E)$ by $U\mapsto (M,U)$ in the usual way.  The image $(M,U)$ is indeed projective: if $U = \limproj{i} U_i$ for modules $U_i$ in $\tn{add}(M)$ then $(M,U) = \limproj{i} (M,U_i)$ and each $(M,U_i)$ is easily seen to be projective.  By \cite[Lemma 5.4.1]{RZ12010}, to check the projectivity of $(M,U)$ we only need to complete diagrams of the form
 $$
\xymatrix{
            & (M,U)\ar[d]  \\
X \ar@{->>}[r] & \,Y}
$$
with the modules $X,Y$ finite, in which case the given map from $(M,U)$ factors through a finitely generated projective quotient, giving the map we require.  We need to show that $\Gamma$ is essentially surjective and fully faithful.

\medskip

To see that $\Gamma$ is essentially surjective, observe that an arbitrary projective $E$-module has the form $P = \bigoplus_{i}\prod_{\kappa_i}(M,M_i)$ by \cite[\S 2.5]{symondspermcom1}.  But now
$$\Gamma(\bigoplus_{i}\prod_{\kappa_i}M_i) = (M,\bigoplus_{i}\prod_{\kappa_i}M_i)\cong \bigoplus_{i}\prod_{\kappa_i}(M,M_i) = P,$$
as required.

We demonstrate next that $\Gamma$ is faithful.  We need to show that if $\alpha:\prod U_i\to V$ is continuous and non-zero ($V$ some
other object of $\tn{Add}(M)$), then the corresponding map $\Gamma(\alpha):(M,\prod U_i)\to (M,V)$ is non-zero.  Consider the open subset $Y=V\backslash\{0\}$ of $V$.  Then the inverse image $X$ of $Y$ under $\alpha$ is a non-empty (because $\alpha\neq 0$) open subset of $\prod U_i$.  But if $X$ is open then it contains an element $x$ of the form $(u_i)$ where $u_i$ is 0 for almost all coordinates.  Let the non-zero
coordinates of $x$ be $i_1,\dots,i_n$.  Now the map
$$U_{i_1}\oplus\dots\oplus U_{i_n}\hookrightarrow \prod U_i \xrightarrow{\alpha} V$$
is non-zero since $\alpha(x)\neq 0$.  Since this leftmost module is a direct sum (coproduct), the corresponding universal property gives that there must be some index ($i_1$, say) so that
$$U_{i_1}\hookrightarrow U_{i_1}\oplus\dots\oplus U_{i_n}\hookrightarrow \prod U_i \xrightarrow{\alpha} V$$
is non-zero.  We have a projection $M\twoheadrightarrow U_{i_1}$.  Now
$$\Gamma(\alpha)(M\onto U_{i_1}\to \prod U_i)\neq 0$$
by construction, and hence $\Gamma(\alpha)\neq 0$.

It remains to show that $\Gamma$ is full.  Given $\gamma:(M,U)\to (M,V)$, we need to find a continuous map
$\alpha:U\to V$ such that $\gamma(-) = \alpha\circ-$.  We split the work into two cases.  Firstly, when $R$ finite and secondly, when $R$ is general.  Write $U = \prod_i U_i$ with each $U_i\in \{M_1,\dots,M_s\}$.

Case 1: Let $V$ be an object of $\tn{add}(M)$, hence finite, and fix $\gamma:(M,\prod_i U_i)\to (M,V)$
continuous.  Then $(M,V)$ is finite since $M$ is finitely generated, so the kernel of $\gamma$ is an open subgroup of
$(M,\prod_i U_i) \cong \prod_i (M,U_i)$.  In particular, almost all $(M,U_i)$ map to 0 under $\gamma$.  Denote by $1,\dots,m$ the indices that don't map to 0 under
$\gamma$.  Restricting $\gamma$ to this finite sum, we have a map
$$\gamma':\bigoplus_{1,\dots,m}(M,U_i)\to (M,V).$$
Now, the version of the result for finitely generated modules (see eg. \cite[Proposition 2.1]{symondsstructure}) tells us that there is a unique $\alpha':\bigoplus_{1,\dots,m}U_i\to V$ such that
$\gamma' = \alpha'\circ-$.  Extend $\alpha'$ to a continuous map $\alpha:\prod_i U_i\to V$ by setting it to be $\alpha'$ on $1,\dots,m$ and 0 elsewhere.  Now given $\rho = (\rho_i)\in \prod_i (M,U_i)$ we have
$$\gamma(\rho) = \prod_i\gamma(\rho_i) = \gamma'(\rho_1 +\dots +\rho_m) + 0 = \alpha'\circ(\rho_1 +\dots +\rho_m) + 0 = \alpha\circ\rho,$$
as required.

Next let $V$ be an arbitrary object of $\tn{Add}(M)$, and write it as $V=\limproj{j} \{V_j, \varphi_{jk}\}$, with the $V_j$ objects of $\tn{add}(M)$. Again we consider some continuous homomorphism
$\gamma:(M,U)\to (M,V)$.  We have that $(M,V)\cong \limproj{j} \{(M,V_j), \varphi_{jk}\circ-\}$.  Fix some $\rho:M\to U$.  Then $\varphi_j\gamma(-):(M,U)\to (M,V_j)$, so
by the previous paragraph we have a unique continuous map $\alpha_j:U\to V_j$ such that $\varphi_j\circ\gamma(-) = \alpha_j\circ-$.  Note that for any $\rho:M\to U$ we have
$$\alpha_j\rho = \varphi_j\gamma(\rho) = \varphi_{jk}\varphi_k\gamma(\rho) = \varphi_{jk}\alpha_k\rho,$$
so that $\alpha_j = \varphi_{jk}\alpha_k$.  That is, the $\alpha_j$ provide a map of inverse systems $U\to \{V_j, \varphi_{jk}\}$, and hence a unique continuous map
$\alpha:U\to V$.  But now $\gamma = \alpha\circ-$ because for all $\rho:M\to U$:
$$\gamma(\rho) = (\varphi_j\gamma(\rho)) = (\alpha_j\rho) = (\varphi_j\alpha\rho) = \alpha\rho.$$
This completes case 1.

Case 2:  Let $\pi$ generate the maximal ideal of $R$.  We are given $\gamma:(M,U)\to (M,V)$.  For each $n\in \N$, this induces a map
$\gamma_n:(M/\pi^nM, U/\pi^nU)\to (M/\pi^nM, V/\pi^nV)$.  By Case 1, $\gamma_n  = \alpha_n\circ-$, for some continuous map $\alpha_n:U/\pi^nU\to V/\pi^nV$.

But we also have that $U\cong\limproj{n} U/\pi^nU$ and $V\cong \limproj{n} V/\pi^nV$.  Let $n\leqslant m$ and note that the diagram
$$
\xymatrix{
U/\pi^mU \ar[r]^{\alpha_m}\ar[d] & V/\pi^mV\ar[d]  \\
U/\pi^nU \ar[r]^{\alpha_n} & V/\pi^nV}
$$
commutes.  It follows that the $\alpha_n$ are a map of inverse systems, so yield a unique continuous map $\alpha:U\to V$.  This map has the property that
$\gamma = \alpha\circ-$.
\end{proof}

\begin{theorem}\label{limits are products}
The categories $\tn{Add}(M)$ and $\limproj{}\tn{add}(M)$ are equivalent.
\end{theorem}

\begin{proof}
Note that $E=\tn{End}_{RG}(M)$ is profinite because $M$ is
finitely generated.  The inclusion $\tn{Add}(M)\to \limproj{}
\tn{add}(M)$ is clearly fully faithful, so we need only check that
it is essentially surjective.  Let $U = \limproj{I} U_i$ be an
object of $\limproj{} \tn{add}(M)$.  Apply the functor $\Gamma =
\Hom_{RG}(M,-)$ to the inverse system for $U$ to get an inverse
system of projective right $E$-modules. The inverse limit is
$(M,U)$, which we observed in the previous proof is a projective
right $E$-module.  The theorem is now immediate from Proposition
\ref{AddM equivalent to ProjE}.
\end{proof}

\bigskip

\begin{corollary}\label{permprod}
Let $G$ be a finite $p$-group, $R$ a profinite local commutative noetherian ring and let $U$ be a  profinite
permutation $R G$-lattice. Then $U$ can be expressed as a
cartesian product of cyclic permutation modules

$$U=\prod_{H\leqslant G}\prod_{\kappa_H}R[G/H],$$
 where
$\kappa_H$ is a cardinal number.
\end{corollary}

\begin{proof}There are finitely many isomorphism classes
of indecomposable permutation modules (indexed by the conjugacy
classes of subgroups of $G$), so the result follows from Theorem
\ref{limits are products}.
\end{proof}

\bigskip
Some notation that will remain in force throughout this section.  Given a finite group $G$ with subgroup $H$, a commutative ring $R$ and a profinite $RH$-module $V$, the profinite $RG$-module $V\ind{G}$ is defined to be $RG\hat\otimes_{RH}V$, where $G$ acts on the left factor as with the usual induced module, and where ``$\hat\otimes$'' denotes the completed tensor product \cite[\S 5.5]{RZ12010}.  The restriction of the $RG$-module $U$ to $RH$ is denoted $U\res{H}$.  When $N$ is a normal subgroup of $G$ and $U$ is an $RG$-module, the $R[G/N]$-modules $U^N, U_N$ denote, respectively, the submodule of $N$-invariants of $U$, and the module of $N$-coinvariants $U/I_NU$ of $U$, where $I_N$ is the kernel of the augmentation map $\varepsilon:RN\to R$.

\bl{gp-rg}\label{inv iso coinv} Let $N$ be a normal subgroup of the finite group $G$, $R$ a profinite integral domain and $U$ a profinite
$RG$-module that is free as an $RN$-module.  The map $\phi:U\to U^N$ defined by
$\phi(u):=\sum_{n\in N}nu$ is an \epi\ of $RG$-modules with kernel $I_NU$.  Thus $U^N\cong U_N$. \el

\begin{proof} The element $\sum_{n\in N}n$ belongs to
the centre of $RG$, hence $\phi$ is a well-defined $RG$-module homomorphism.  The kernel of $\phi$ clearly contains $I_NU$.  We can check the other assertions after first restricting to $N$.  Write $U\res{N}$ as a product of modules of the form $RN$, indexed by a set $I$ .  For a subset $X$ of $I$ let $U_X$ be the subproduct of $U$ indexed by $X$ and when $Y\subseteq X$ are finite subsets of $I$, let $\rho_{XY}: U_Y\to U_X$ be the obvious projection.  Restricting the inverse system $\{U_X, \rho_{XY}\}$ to $U^N$ yields $U^N = \limproj{X}\{(U_X)^N,\rho_{XY}\}$.  For each finite $X$ the obvious map $\phi_X:(U_X)_N\to (U_X)^N$ is the component at $X$ of an isomorphism of inverse systems \cite[Lemma 4]{bulletin}, and hence $\phi$ is an isomorphism.\end{proof}

\medskip

Let $\zeta$ be a primitive $p$th root of unity and $\Z_p(\zeta)$ the corresponding totally ramified extension of $\Z_p$ of degree $p-1$.  Let $\pi$ be a generator of the unique maximal ideal of $\Z_p(\zeta)$.  The following may be viewed as a special case of a generalization of \cite[Theorem 3]{weiss} to infinitely generated modules.

\begin{lemma}\label{rootunity} Let $G$ be a finite $p$-group and $V$ a
$\Z_p(\zeta)G$-lattice. Suppose that $V/\pi V$ is a free $\F_pG$-module. Then $V$ is a free $\Z_p(\zeta)G$-module.\end{lemma}

\begin{proof} Let $\phi:V\longrightarrow V/\pi V$ be the natural epimorphism.
By \cite[Proposition 2.2.2]{RZ12010},
$\phi$ admits a continuous section $\delta: V/\pi V \longrightarrow V$ with $\delta(\pi V)=0$. Consider a profinite
space $\Omega$ of free generators of $V/\pi V$ converging
to $0$. Put $\mathcal{X}=\delta\left( \Omega \right)$.  Let $A$ be a free pro-$p$ $\Z_p(\zeta)G$-module on
$\mathcal{X}$ and let $f:A\longrightarrow V$ be the
$\Z_p(\zeta)G$-homomorphism induced by sending $\mathcal{X}$
identically to its copy in $V$. Then as a pro-$p$
$\Z_p(\zeta)$-module, $A$ is free pro-$p$ on the basis
$G\mathcal{X}$. Since $V/\pi V$ is a free $\mathbb{F}_p G$-module
on $\Omega$, it is a free $\F_p$-module on $G\Omega$.

Notice that as  $\Z_p(\zeta)$-modules, the radicals of $A, V$ respectively are $\pi A, \pi V$, and that the map $\bar f: A/\pi A\to V/\pi V$ is an isomorphism, so that $f$ is an isomorphism by \cite[Lemma 7.4.4]{W}, as required.
\end{proof}

\medskip

As in \cite{weiss}, a finitely generated generalized permutation $\Z_p(\zeta)G$-lattice is a finite direct sum of modules of the form $\varphi\ind{G}$, where ``$\varphi$'' represents a rank one $\Z_p(\zeta)H$-lattice, $H$ some subgroup of $G$, with action from $H$ coming only from a group homomorphism $\varphi:H\to \langle\zeta\rangle$.  A profinite generalized permutation $\Z_p(\zeta)G$-lattice is an inverse limit of finite rank permutation $\Z_p(\zeta)G$-lattices.  Since $G$ is finite, there are only finitely many isomorphism types of indecomposable finite rank generalized permutation $\Z_p(\zeta)G$-lattice.  Thus, by Theorem \ref{limits are products}, profinite generalized permutation $\Z_p(\zeta)G$-lattices are simply products of indecomposable finite rank generalized permutation $\Z_p(\zeta)G$-lattices.

\medskip

We give one more technical lemma before proving the main result of this section.

\begin{lemma}\label{indexing sets coincide}
Let $W$ be a second countable profinite $\F_pG$-module that can be
expressed as a product of finitely generated indecomposable
submodules. Let $\prod_{s\in S}\prod_{\kappa_s}W_s$ and
$\prod_{s\in S}\prod_{\nu_s}W_s$ be two continuous decompositions
of $W$, where the $W_s$ are pairwise distinct finitely generated
indecomposable $\F_pG$-modules and $\kappa_s, \nu_s$ are
cardinals.  Then for each $s\in S$ we have $\kappa_s=\nu_s$.
\end{lemma}

\begin{proof}
Fix $s$, suppose that $\kappa_s=n$ for some finite number $n$ and suppose for contradiction that $\nu_s$ is strictly greater than $n$.  Then, by considering the second decomposition, we obtain a continuous and split projection onto a product of $n+1$ copies of $W_s$.  Let $X$ be the kernel of this projection.  Since $X$ is open in $W$, we can find some cofinite subset $I$ of the indexing set for the first decomposition such that the product $Y$ indexed by $I$ is contained in $X$.  We obtain a continuous surjection $W/Y\to W/X$, which splits since the map $W\to W/X$ splits.  Thus $W/X$ is a direct summand of $W/Y$.  But the former is a product of $n+1$ copies of $W_s$, while the latter has at most $n$ direct summands isomorphic to $W_s$, contradicting the Krull-Schmidt theorem for finitely generated $\F_pG$-modules.

Meanwhile, if $W_s$ appears infinitely many times in both compositions then $\kappa_s = \nu_s = \aleph_0$ since $W$ is second countable.
\end{proof}

\begin{theorem}\label{infgen}
Let $G$ be a finite $p$-group and let $U$ be a second countable profinite $\Z_p
G$-lattice.  Suppose that there is a central subgroup $N$ of $G$
of order $p$ such that
 \begin{itemize}
  \item The restriction $U\res{N}$ is a free $\Z_p N$-module,
  \item $U^N$ is a permutation $\Z_p[G/N]$-module
    \item $U/U^N$ is a generalized permutation $\Z_p(\zeta)G$-lattice.
 \end{itemize}
Then $U$ itself is a permutation $\Z_pG$-lattice.
\end{theorem}

\begin{proof}
Let $N = \langle n\rangle$.  Consider the ideals $I_N  = \ker(\varepsilon:\Z_pN\to \Z_p) = \langle n-1, n^2-1,\cdots, n^{p-1}-1\rangle$ and $\Z_pN^N = (1+n+\cdots+n^{p-1})\Z_p$ of $\Z_pN$.  As in \cite{weiss}, we have a pullback diagram
 $$
\xymatrix{
\Z_pN \ar[r]\ar[d] & \Z_pN/\Z_pN^N\ar[d]  \\
\Z_pN/I \ar[r] & \Z_pN/(I+\Z_pN^N).}
$$
Fix an isomorphism $\psi:N\to \langle\zeta\rangle$.  Via $\psi$ we have
an isomorphism $Z_pN/Z_pN^N\cong \Z_p(\zeta)$, and the above pullback can be expressed as
 $$
\xymatrix{
\Z_pN \ar[r]\ar[d] & \Z_p(\zeta)\ar[d]  \\
\Z_p \ar[r] & \,\F_p.}
$$

Apply the functor $-\ctens_{\Z_pN} U$ to the above diagram.  Since
$U\res{N}$ is free, we can write it as an inverse limit of finite
rank free $\Z_pN$-modules.  Since inverse limits are functorial
and commute naturally with completed tensor products, the argument
in \cite{weiss} applied to the finite rank quotients shows that
our diagram is an inverse limit of pullbacks, and is hence itself
the pullback
$$
\xymatrix{
U \ar[r]\ar[d] & U/U^N\ar[d]  \\
U_N \ar[r] & \overline{U_N} = \overline{U/U^N}.}
$$

By  \preflemma{gp-rg} we have that $U_N\cong U^N$, so that $U_N$
is a permutation module by assumption, and hence so too is
$\overline{U_N}$.

Thus we have a permutation module $U_N$ and a generalized
permutation module $U/U^N$, which by Theorem \ref{limits are products} are products of indecomposable modules of the same type.  The indexing sets of the
indecomposable summands of these modules both biject naturally
onto corresponding indexing sets of indecomposable summands of
$\overline{U_N}$, and hence by Lemma \ref{indexing sets coincide} their indexing sets coincide.  Now as in \cite{weiss}, we can write
$$U_N \cong \prod_r \Z_p\ind{G}_{G_r},\,\, U/U^N \cong \prod_r\varphi_r\ind{G}_{G_r}$$
where $\varphi_r:G_r\to \langle \zeta\rangle$ is a group
homomorphism whose restriction to $N$ is $\psi$.  Let
$H_r=\ker(\varphi_r)$.  Again following \cite{weiss} we obtain
that
 $$
\xymatrix{
\Z_p\ind{G}_{H_r} \ar[r]\ar[d] & \varphi_r\ind{G}_{G_r}\ar[d]  \\
\Z_p\ind{G}_{G_r} \ar[r] & \F_p\ind{G}_{G_r}}
$$
is a pullback.  The product of these pullbacks as $r$ varies is
itself a pullback diagram, whose upper left component is the
permutation module $L = \prod_r \Z_p\ind{G}_{H_r}$.  We want to
show that $L\cong U$.  The discussion above gives isomorphisms
$h:L_N\to U_N$ and $f:L/L^N\to U/U^N$ (of $\Z_pG$ and
$\Z_p(\zeta)G$-modules respectively).  Let $h',f'$ denote the
isomorphisms $\overline{L_N}\to \overline{U_N}$ induced by $h,f$
respectively, so that $\overline{h}^{-1}\overline{f}$ is an
$\F_pG$-module automorphism of $\overline{L_N}$.  By
\cite[Lemma 3.23]{symondspermcom1} we can lift this automorphism to a
$\Z_pG$-module endomorphism $k$ of $L_N$.  Since $L_N$ is a free
$\Z_p$-module, \cite[7.4.4, 2.5.2]{W} tells us that $k$ is an
automorphism.   Replace the isomorphism $h$ by the isomorphism
$hk$ and note that $\overline{hk}=\overline{f}$.  Now a diagram
chase gives the required isomorphism $L\to U$, completing the
theorem.
\end{proof}

\medskip

\begin{proof}[of Theorem \ref{infgenfree}]
Suppose first that $N$ has order $p$.  We maintain the notation from the
proof of Theorem \ref{infgen}.  The second hypothesis implies that
$\overline{U^N}$ is a free $\F_p G$-module, and now Lemma
\ref{rootunity} tells us that $U/U^N$ is a free
$\Z_p(\zeta)G$-lattice.  The argument in Theorem \ref{infgen} thus
shows that $U$ is a pullback of free modules (the second countable
hypothesis is not required because we do not need to use Lemma
\ref{indexing sets coincide}).  The result now follows immediately
from the uniqueness of pullbacks.

We now proceed by induction on the order of $N$.  Let $K$ be a normal subgroup of $G$ contained as a subgroup of index $p$ in $N$.  The module $U_K$ has the properties that $U_K\res{N/K}$ and $(U_K)^{N/K}\cong U^N$ (see Lemma \ref{inv iso coinv}) are free $\Z_p[G/K]$-modules, so that $U_K$ ($\cong U^K$) is free by the first paragraph.  Since the order of $K$ is less than the order of $N$, by induction we see that $U$ is free.
\end{proof}

\section{HNN-embedding}

For the convenience of the reader we recall some concepts and
terminology from \cite{HZ07}, adapting the definitions to the category of
pro-$p$ groups.

\bd{boolean sp} A {\em boolean} or {\em profinite} space is, by
definition, an inverse limit of finite discrete spaces, i.e. a
compact, Hausdorff, totally disconnected topological space.
Morphisms in the category of boolean spaces are continuous maps.
\ed

A profinite space $X$ with a profinite group $G$ acting
continuously on it will be called a $G$-{\em space}.

\bd{TG}\label{finite subgroups} An inverse system of finite groups with projective limit $G$ induces an inverse system of the sets of subgroups of the groups in the inverse system, whose limit is the space of closed subgroups of $G$.  If $G$ is virtually torsion free, the subspace $\Fin G$ of non-trivial finite subgroups of $G$ is closed, hence inherits a natural profinite topology (the {\em subgroup topology}).  Equipped with this topology, $\Fin G$ with $G$ acting by conjugation becomes a $G$-space.\ed

We recall the notion of a pro-$p$ \HNNgrp\ as defined in \pcite{HZ07} for the pro-$\cC$ case.

\bd{sheaf} A {\em sheaf} of pro-$p$ groups (over a profinite space $X$)
is a triple $(\cG,\gamma,X)$, where $\cG$ and $X$ are profinite
spaces, and $\gamma$ is a continuous map from $\cG$ onto $X$,
satisfying the following two conditions:
\begin{rmenumerate}
\item for every $x\in X$, the fiber $\cG(x)=\gamma^{-1}(x)$ over
      $x$ is a pro-$p$ group;
\item if $\cG^2$ denotes the subspace of $\cG\times\cG$ consisting
of
      pairs $(g,h)$ such that $\gamma(g)=\gamma(h)$, then the map
      $\mu_{\cG}:\cG^2\longrightarrow \cG$, defined by $\mu_{\cG}(g,h):=
      g^{-1}h \in \cG(\gamma(g))=\cG(\gamma(h))\subseteq\cG$,
      is continuous.
\end{rmenumerate}

If there is no danger of confusion we shall write $(\cG,X)$
instead of $(\cG,\gamma,X)$.

A {\em morphism} of sheaves of pro-$p$ groups
$(\alpha,\bar\alpha):(\cG,\gamma,X) \rightarrow ({\cal H},\eta,Y)$ is a pair  of continuous
maps $\alpha: \cG \rightarrow {\cal H}$,
$\bar\alpha:X\rightarrow Y$ such that the diagram

$$\xymatrix{
\cG\ar@{->}^\alpha[rr]\ar@{->}^\gamma[d] & &{\cal H}\ar@{->}^\eta[d]\\
X\ar@{->}^{\bar\alpha}[rr] & &Y }$$ is commutative and, for all
$x\in X$, the restriction $\alpha_x:=\alpha_{|\cG(x)}$ of $\alpha$
to the fiber $\cG(x)$ is a homomorphism from $\cG(x)$ to
${\cal H}(\bar\alpha(x))$.

In the special case when $Y=\{y\}$ consists of a single element
set, we obtain with $H:={\cal H}(y)$ the definition of a {\em fiber
morphism} $\alpha:\cG\longrightarrow H$, of the sheaf $\cG$ of
\pCgrp s to the \pCgrp\ $H$. We shall say that $\alpha$ is a {\em
fiber \mono} if $\alpha_x$ is injective for every $x\in X$.\ed

The simplest example of a sheaf of \pCgrp s is that of the {\em
constant sheaf}
 $(G\times X,{\rm pr}_X,X)$, where $G$ is some \pCgrp\ and
${\rm pr}_X:G\times X\longrightarrow X$ is the projection. For
every $x\in X$, the fiber $(G\times X)(x)=G\times \{ x\}$ is
isomorphic to $G$.

\bd{free product} Given a sheaf $(\cG,X)$ of \pCgrp s, the pro-$p$ product $G = \coprod_{x\in X}\cG(x)$ is a pro-$p$ group $G$
together with a fiber homomorphism
$\upsilon:(\cG,X)\longrightarrow G$, having the following universal
property: for every \pCgrp\ $K$ and every fiber homomorphism
$\beta:(\cG,X)\rightarrow K$, there exists a unique homomorphism
\[\omega:G\longrightarrow K,\]
such that  $\beta=\omega\upsilon$.\ed

When considering free pro-$p$ products, we make frequent use of the following:

\bt[(\cite[Theorem 9.1.12]{RZ12010} and \cite[Theorem 4.2]{RZ2}]\label{conjugacy}
Let $G=\coprod_{i=1}^n G_i$ be a free
profinite (pro-$p$) product. Then $G_i\cap G_j^g=1$ for either
$i\neq j$ or $g\not\in G_j$.

Every finite subgroup of $G$ is conjugate to a subgroup of
one of the factors $G_i$.
\et

Next we state the pro-$p$ analogue of the concept of an HNN-group,
given in \cite{HZ07} for pro-$\C$-groups.

\bd{HNN-grp} Let $H$ be a pro-$p$ group and
$\partial_0,\partial_1:(\cG,T)\rightarrow H$  fiber \mono s. A
specialization into $K$ consists of a
homomorphism $\beta:H\longrightarrow K$ and a continuous map
$\beta_1: T\longrightarrow K$ such that for all $t\in T$ and $g\in
\cG(t)$ the equality $\beta(\partial_{0}(h))= \beta_1(t)^{-1}\beta(\partial_{1}(h))\beta_1(t)$
is valid. Such a specialization into $K$ will be denoted $(\beta,\beta_1):(H,\cG,T)\to K$.

The pro-$p$ HNN-group  is then a pro-$p$ group $G$ together with
a specialization $(\upsilon,\upsilon_1):(H,\cG,T)\longrightarrow
G$, with the following universal property: for every
\pCgrp\ $K$ and every specialization
$(\beta,\beta_1):(H,\cG,T)\rightarrow K$, there exists a unique
homomorphism
\[\omega:G\longrightarrow K,\]
such that $\omega\upsilon_1=\beta_1$ and $\beta=\omega\upsilon$. We shall
denote the pro-$p$ HNN-group by $\HNN(H,\cG,T)$.  The group $H$ is called the {\em
base group},  and elements
$t\in T$ are  called the {\em stable letters}.
\ed
\medskip

Note that for $T$ a singleton set, identifying $\cG(t)$ with its
image under $\partial_0$ and setting $f:=\partial_1$, the
definition of a pro-$p$-\HNN\ extension given in \cite[Section
9.4]{RZ12010} is recovered.  By \cite[Proposition
9]{HZ07}, the pro-$p$ HNN-group $G=\HNN(H,\cG,T)$ exists and is
unique.

\medskip
A pro-$p$ HNN-group is a special case of the fundamental pro-$p$
group $\Pi_1(\cG,\Gamma)$ of a \pgraph\  of \pCgrp s
$(\cG,\Gamma)$ as introduced in \cite{ZM90}. Namely, a pro-$p$
HNN-group can be thought as $\Pi_1(\cG,\Gamma)$, where $\Gamma$ is
a connected profinite graph having just one
vertex.

For the rest of this section let $G$ be a  second countable
\vfppgrp, and fix an open \fpp\ normal subgroup $F$ of $G$ of
minimal index. Also suppose that the centralizer $C_F(t)=\ugp$ for
every torsion element $t\in G$. Let $K:=G/F$ and form the \fpp\
product $G_0:=G\amalg K$. Let $\psi:G\to K$ denote the canonical
projection. It extends to an \epi\ $\psi_0:G_0\to K$, by sending
$g\in G$ to $gF\in K$ and each $k\in K$ identically to $k$, and
extending to $G_0$ using the universal property of the \fpp\
product. Notice that the kernel $L$ of $\psi_0$ is an open
subgroup of $G_0$ and, since $L\cap G=F$ and $L\cap K=\ugp$, the
\pp\ version of the \KST\, \cite[Theorem 4.3]{M89} tells us that
$L$ is \fpp. By \cite[Lemma 5.6.7]{RZ12010} there exists  a
continuous section $Fin(G)/G\longrightarrow Fin(G)$.  Define the
subspace $\Sigma$ of $Fin(G_0)$ to be the union of the image of
this section with the subgroup $K$. Since by Theorem
\ref{conjugacy} all finite subgroups of $G_0$ can be conjugated
either into $G$ or into $K$, the natural map
$Fin(G_0)\longrightarrow Fin(G_0)/G_0$ restricted to $\Sigma$ is a
homeomorphism. Define a sheaf $(\cG,\Sigma)$ by setting
$\cG=\{(g,S)\in G_0\times \Sigma\mid g\in S\}$ and defining
$\gamma:\cG\longrightarrow \Sigma$ to be the restriction to $\cG$
of the natural projection $G_0\times \Sigma\longrightarrow
\Sigma$. Define an equivalence relation on $\Sigma$ by putting
$S_1\sim S_2$ if $S_1$ and $S_2$ are contained in the same maximal
finite subgroup of $G_0$.  This defines an equivalence relation
since maximal finite subgroups have trivial intersection
 by \cite[Lemma 9]{bulletin}, and it is easy to see that $\sim$ is
closed (indeed if $K_1, K_2$ are maximal finite subgroups
containing $S_1$ and $S_2$ respectively and such that $K_1\cap
K_2=1$, then there is a finite quotient of $G_0$ where the images
of $K_1$ and $K_2$ intersect trivially as well). Put
$T=\Sigma/\sim$. Given $t\in T$, denote by $K_t$ the unique
maximal finite subgroup corresponding to $t$ (i.e. if $S\in
\Sigma$ is a representative of an equivalence class $t$ then $K_t$
is the unique maximal finite subgroup containing $S$). Let
$(\cK,T)$ be the subset of $G_0\times T$ defined by
$(\cK,T)=\{(k,t)\mid k\in K_t, t\in T\}$.

\begin{lemma}  $(\cK,T)$ is a subsheaf of $G_0\times T$.\end{lemma}

\begin{proof} We show only that $\cK$ is closed in $G_0\times T$, since the rest follows easily. Let $(g,t)\in G_0\times T$ be such that
$g\not\in K_t$. Note that  the set of non-trivial torsion elements
is closed in $G_0$, because they are of bounded order and can not
have $1$ as an accumulation point since $G$ is virtually
torsion-free. Therefore there is an open torsion free subgroup $U$
such that $(gU\times T)\cap \cK=\emptyset$ as needed. \end{proof}

\medskip

With notation as above we define the pro-$p$ \HNNgrp\ $\tG=HNN(G_0,\cK,T)$ by setting $\partial_0$ to
 be the fiber homomorphism that sends each $K_t$ identically to its
 copy in $G$ and setting $\partial_1(k_t)=\psi(k_t)$ for every $k_t\in K_t$. Here $G_0$ is the base group, the $K_t$ are
associated subgroups, and $T$ is a set of stable letters in
the sense of \prefdef{HNN-grp}.

\medskip

The objective of this section is to show that the centralizers of
torsion elements of $\tilde G$ are finite. It was already proved
in \cite[Theorem 12]{HZ07} that $G$ embeds into $\tilde G$ and
that the torsion elements of $\tilde G$ are conjugate to elements
of $K$ (in fact, \cite[Theorem 12]{HZ07} has a slightly different
statement, namely that the space $T$ is a subspace of $\Fin G$,
but this was not used in the proof).  Thus, in the following lemma
we need to prove only item (iv), since items (i)-(iii) are the
subject of the proof of \cite[Theorem 12]{HZ07}.

\bl{centr} Let $\tilde G= \HNN(G_0,K_t,\phi_t,T)$ be as explained above, and let $\tilde F$ be the kernel of the map $\tilde G\to K$ induced from the universal property of pro-$p$ HNN groups -- an open normal free pro-$p$ subgroup of $\tilde G$.  Then

\begin{enumerate}
 \item
 [(i)]  $G_0$ and therefore $G$ canonically embed in $\tilde G$.
 \item [(ii)] $\tilde G = \tilde F\rtimes K$ is a semidirect product.
 \item [(iii)] In $\tilde G$ every finite subgroup is conjugate to a subgroup of  $K$.
 \item[(iv)] $C_{\tilde F}(g)=1$ for every torsion
element $g\in \tilde G$.\end{enumerate}

\el

\begin{proof} (iv) There is a standard pro-$p$ tree $S:=S(\tilde G)$
associated to $\tilde G:= \HNN(G_0,K_t,\phi_t,T)$ on which
$\tilde G$ acts naturally such that the vertex stabilizers are
conjugates of $G_0$ and each edge stabilizer is a conjugate of
some $K_t$ (cf. \pcite{RZ2} and \cite[\S 3]{ZM90}).

\medskip

\noindent{\em Claim:
Let $e_1,e_2$ be two edges of $S$ with a common vertex
$v$ which is not the terminal vertex of both of them. Then the intersection of the stabilizers $\tilde
G_{e_1}\cap\tilde G_{e_2}$ is trivial.
}

\medskip

\claimproof By translating $e_1,e_2, v$ if necessary we may assume
that $G_0$ is the stabilizer of $v$. Then  we
have two cases:

1) $v$ is the initial vertex of $e_1$ and $e_2$. Then $\tilde
G_{e_1}=K_t^g$ and $\tilde G_{e_2}=K_{t'}^{g'}$ with $g,g'\in G_0$
and either $t\neq t'$ or $g\not\in K_tg'$ and by construction of
$\tilde G$ we have $K_t^g\neq K_{t'}^{g'}$ if $t\neq t'$. Suppose
that $K_t^g\cap K_{t'}^{g'}\neq\ugp$. Then, since $G_0=G\amalg K$,
we may apply Theorem \ref{conjugacy}, in order to deduce the
existence of $g_0\in G_0$ with $K_t^{gg_0}\cap K_{t'}^{g'g_0}\le
G_0$. Now by \cite[Lemma 9]{bulletin} two maximal finite
subgroups of $G_0$ have trivial intersection. So we have
$K_t^g\cap K_{t'}^{g'}=\ugp$, as needed.
\medskip

2) $v$ is the terminal vertex of $e_1$ and the initial vertex of
$e_2$. Then $\tilde G_{e_1}=K^g$ and $\tilde G_{e_2}=K_t^{g'}$ for
$g,g'\in G_0$ so they intersect trivially by the definition of
$G_0$ and Theorem \ref{conjugacy}. So the claim holds.

\bigskip

Now pick a torsion element $x\in \tilde G$ and a non-trivial
element $f\in\tilde F$ with $x^f=x$. Let $e$ be an edge of $S$
stabilized by $x$. Then $fe$ is also stabilized by $x$ and, since
by \cite[Theorem 3.7]{RZ2} the fixed set $S^x$ is a subtree, the
path $[e,fe]$ is fixed by $x$ as well. Note that a common vertex
of $e$ and $fe$ (if it exists) can be terminal to only one of them,
since $f$ can not stabilize any vertex. So $[e,fe]$ has at least
one vertex which is initial for its incident edge. Thus by the
claim $x=1$.\end{proof}

\section{Proof of the main theorem}

We require the following before turning to the key proposition.

\begin{lemma}\label{group-module connection} Let $G$ be a semidirect product of a \fppgrp\ $F$ with
a finite $p$-group\ $K$ such that every torsion element is
conjugate to a subgroup of $K$. Then
\begin{enumerate}
 \item
 [(i)] $G=(K_G)\rtimes F_0$, where $K_G$ is the normal closure of $K$ in $G$ and $F_0$ is free pro-$p$.
 \item[(ii)] The natural homomorphism $G\longrightarrow F_0$
 induces a natural epimorphism $M=F/[F,F]\longrightarrow
 F_0/[F_0,F_0]$ of $\Z_pK$-modules with kernel $I_KM$.
 \end{enumerate}
 \end{lemma}

 \begin{proof} For any group $X$ we denote by $\tor{X}$ the set of torsion elements of $X$. By \cite[Proposition 1.7]{Z2003} $\torfactor G$ is free pro-$p$, so we can fix a section $F_0$ of $\torfactor G$ in $G$, proving
(i). In fact, we can choose $F_0$ to live inside $F$, as we may
since the restriction of $G\longrightarrow \torfactor G$ to $F$ is
an epimorphism.  Thus we can write $F/[F,F]=F_0/[F_0,F_0]\oplus L$
as a direct sum of $\Z_p$-modules, for some $\Z_p$-submodule $L$.
We claim that a profinite $\Z_p$-basis of $F_0/[F_0,F_0]$ is a
profinite free $\Z_pK$-basis of $U$.  Again note that $\torgp{G}$
coincides with the normal closure $K_G$ of $K$ in $G$. Thus we
have the following commutative diagram
$$\xymatrix{G\ar[r]\ar[d]& G/K_G\ar[d]\\
          \bar G=F/[F,F]\rtimes K\ar[r]& \bar G/K_{\bar G}}$$
of natural epimorphisms. The lower horizontal map restricted to
$F/[F,F]$ has kernel $[F/[F,F],K]$, which coincides with $I_KU$
($I_K$ the augmentation ideal of $\Z_pK$) when we view $U=F/[F,F]$
as a $\Z_pK$-module.\end{proof}

\bp{model1} Let $G$ be a semidirect product of a \fppgrp\ $F$ with
a finite $p$-group\ $K$ such that every torsion element is conjugate to a
subgroup of $K$. Suppose that $C_F(t)=\ugp$ holds for every
torsion element $t\in G$. Then $G=K\amalg F_0$ for a \fpp\ factor
$F_0$. \ep

\begin{proof}
Suppose that the proposition is false. Then there is a \cex\ with
$K$ having minimal order.  By \cite[Theorem 1.2]{Z2003}, when $K$
has order $p$ we have that $G=\coprod_{x\in X} (C_p\times
H_x)\amalg H_0$, where $H_x, H_0$ are free pro-$p$, therefore the
proposition is satisfied, and so we can suppose that $K$ is of
order at least $p^2$.

Let $H$ be a central subgroup of $K$ of order $p$. Then $F\rtimes
H$ satisfies the premises of the proposition and hence $F\rtimes
H$ is of the form $H\amalg F_1$ for some free factor $F_1$.  Let us
denote by bar passing to the quotient modulo the normal closure $H_G$ of $H$ in $G$.  By \cite[Proposition
1.7]{Z2003}
 $\bar F$ is free pro-$p$ and  $\overline{\tor{G}}=\tor{\bar
 G}$. It follows that $\bar G=\bar F\rtimes (K/H)$ and
 every torsion element in $\bar G$ can be lifted to a
conjugate of an element in $K$. So every torsion element in $\bar
G$ is conjugate to an element of $\bar K=K/H$ and we  deduce from
Theorem \ref{conjugacy} that $H_G=\torgp{FH}$. Thus  by the
minimality of $K$ we have \be{FH} \bar G=\bar K\amalg \bar F_0.
\ee for some free pro-$p$ group $\bar F_0$.

 Consider $U:=F/[F,F]$ as
a $\Z_pK$-module and let $I_H$ denote the augmentation ideal of
$\Z_pH$. Since $F\rtimes H=H\amalg F_1=\left(\coprod_{h\in
H}F_1^h\right)\rtimes H$, $H$ acts by permuting the free factors
$F_1^h$, so that $U$ is a free $\Z_pH$-module.
Passing in
\prefeq{FH} to the quotient modulo the commutator subgroup of
$\bar F=(F_0)_{\bar G}$,  one can see using Lemma
\ref{group-module connection} (ii) that $U/I_HU$ is a free
$\Z_p\bar K$-module.
%So by Lemma \ref{rootunity}, $U/U^H$ is a
%free $\Z_p(\zeta)\bar K$-module.
Now an application of Lemma \ref{inv iso coinv} and
Theorem \ref{infgenfree} shows that $U$ itself is
a free $\Z_p K$ lattice.

%We shall  show that $U$ is a free $\Z_p K$-module. Indeed, if any
%summand is not free, then it has the form $\Z_p[K/S]$ for some
%non-trivial subgroup $S$ of $K$, so that the restriction to $S$
%has a trivial direct summand.  Let $S'$ be a proper non-trivial
%subgroup of $K$ such that the restriction $U\res{S'}$ has a
%trivial summand. By the minimality of $K$ we have $FS'=F\rtimes
%S'=F_{S'}\amalg S'$, which implies that $U$ is a free
%$\Z_pS'$-module, a contradiction. Thus $U$ is a free
%$\Z_pK$-module.

 By part (ii) of Lemma \ref{group-module connection}, a closed free basis of
$F_0/[F_0,F_0]$ is a free $\Z_pK$ basis of $U$; this follows from
the fact that a closed subset of $U$ is a free $\Z_pK$-basis if
and only if its image in $U/I_KU$ is a free $\Z_p$-basis.

Consider $\tG:=K\amalg \tF_0$ with $\tF_0\cong \torfactor G$,
and define a \homo\ $\phi:\tG\to G$ by sending $K$ to $K$, $\tF_0$
to $F_0$ and extending to $\tG$ using the universal property of
the free pro-$p$ product.  The images under the Frattini quotient of $K$
and $F_0$ together generate the image of $G$ (since the image of
$K$ is equal to the image of $\torgp{G}$), and hence $G$ is
generated by $K$ and $F_0$.  Thus $\phi$ is an \epi.  Let $\tilde
F=\phi^{-1}(F)$, and note that as the subgroup $\tilde F$  of
$\tilde G$ is the kernel of a homomorphism onto $K$ that restricts
to the identity on $K$, we have $\tilde G = \tilde F\rtimes K$.
Since $\tilde F\rtimes K=K\amalg \tilde F_0=\left(\coprod_{k\in
K}\tilde F_0^k\right)\rtimes K$, $K$ acts by permuting the free
factors $\tilde F_0^k$, so that $\tilde F/[\tilde F,\tilde F]$ is
a free $\Z_pK$-module with basis the image of the basis of $\tilde
F_0$ in $\tilde F/[\tilde F,\tilde F]$.  On the other hand from
the paragraphs above we know that $U$ is a free module with basis
the image of the basis of $F_0$ in $U$. Therefore the kernel of
$\phi$ must be contained in $[\tF,\tF]$. But since $\tilde F$ and
$F$ are free this implies $\tilde F\cap \ker\phi=\ugp$. Since
$K\cap\ker\phi=\ugp$, we conclude that $\phi$ is an isomorphism,
as required. \end{proof}

\begin{proof}[of Theorem \ref{main}]
First form $\tG$ as described before \preflemma{centr}, in order to embed $G$ into a group $\tG$ whose finite subgroups have finite centralizers (by \preflemma{centr}), and, moreover, which has a single conjugacy class of maximal finite subgroups.  By
\prefprop{model1} the group $\tG$ is of the form $\tG=K\amalg F_0$
where $K$ is finite and $F_0$ is \fpp. One now deduces the result from the pro-$p$ version of the Kurosh subgroup theorem \cite[Theorem 4.3]{M89}.
\end{proof}

\section{Acknowledgments}

We express a big thank you to Peter Symonds, whose input has been of
great help to this project.


\begin{thebibliography}{}



\bibitem{Ch94}
         Z. Chatzidakis,
         {Some remarks on profinite HNN extensions},
         Isr. J. Math. {\bf 85}, No.1-3, (1994) 11-18.

\bibitem{efrat} I. Efrat, {Valuations, Orderings, and Milnor
K-Theory}, Mathematical Surveys and Monographs  V. 124, 2006.

         \bibitem{E96} I. Efrat, {On virtually projective groups},
  Michigan Math. J. (3) {\bf 42} (1995) 435--447.

  \bibitem{H93} D. Haran,
  {On the cohomological dimension of Artin-Schreier structures},
  J.\, Algebra (1) {\bf 156} (1993) 219--236.



\bibitem{bulletin}  W.N. Herfort and P.A. Zalesskii,
     {Virtually free \ppgrp s whose torsion elements have finite centralizers,}
     Bull. London Math. Soc. 2008 {\bf 40} (2008) 929--936.

\bibitem{HZ07}  W.N. Herfort and P.A. Zalesskii,
     {Profinite HNN-constructions,}
     J. Group Theory {\bf 10}(6) (2007) 799--809.

 \bibitem{HZ}  W.N. Herfort and P.A. Zalesskii,  {Cyclic extensions of free pro-p groups,} J. Algebra 216 (1999), no. 2, 511-547.

\bibitem{M89} O.V. Melnikov,
    {Subgroups and Homology of Free Products of Profinite Groups},
    {Math. USSR Izvestiya}, {\bf 34}, 1, (1990), 97-119.


\bibitem{RZ12010} L. Ribes and P.A. Zalesskii,
   {Profinite Groups},
   Springer 2010.

\bibitem{RZ2} L. Ribes and P.A. Zalesskii,
    {Pro-$p$ Trees},
   \emph{New Horizons in pro-$p$ Groups} (eds M du Sautoy, \segal\ and \shalev), {\rm Progress in Mathematics 184} (Birkh\"auser, Boston, 2010).

   \bibitem{symondspermcom1} P.A. Symonds, Permutation Complexes for Profinite Groups, Commentarii Mathematici Helvetici, (2007), {\bf 82},
   1-37.

    \bibitem{symondsstructure} P.A. Symonds,
    {Structure Theorems Over Polynomial Rings,}
    {\em Advances in Mathematics} (2007), {\bf 208}, 408-421.

\bibitem{weiss} A. Weiss , Rigidity of $p$-adic $p$-torsion, Ann. of Math. (2) 127 (1988) 317-332.

   \bibitem{W} J.S. Wilson. Profinite Groups. Oxford Science Publications 1998.

\bibitem{Z2003}
   P.A. Zalesskii,
   On virtually projective groups,
   \emph{J. f\"ur die Reine und Angewandte Mathematik (Crelle's Journal) }{572} (2004) 97--110.

\bibitem{ZM90} P.A. Zalesskii and O.V. Melnikov,
  {Fundamental Groups of Graphs of Profinite Groups,}
  {\em Algebra i Analiz} {\bf 1} (1989);
  translated in: {\em Leningrad Math. J.} {\bf 1} (1990), 921-940.


\end{thebibliography}
\end{document}